\documentclass{head/cernrep}
\usepackage{bbm}
\usepackage{xcolor}
\usepackage{mathrsfs}
\usepackage{amsthm}
\usepackage{natbib}
\usepackage{mathtools}
\usepackage{enumerate}
\usepackage{hyperref}
\usepackage{comment}
\usepackage{ulem}

\newcommand{\dif}{\mathrm{d}}
\newcommand{\bP}{\mathbb{P}}
\newcommand{\bG}{\mathbb{G}}
\newcommand{\bE}{\mathbb{E}}
\newcommand{\bR}{\mathbb{R}}
\newcommand{\bN}{\mathbb{N}}
\newcommand{\bS}{\mathbb{S}}
\newcommand{\eps}{\varepsilon}

\newcommand{\cS}{\mathcal{S}}
\newcommand{\cF}{\mathcal{F}}
\newcommand{\cG}{\mathcal{G}}

\theoremstyle{remark}
\newtheorem{remark}{Remark}[section]
\theoremstyle{plain}
\newtheorem{proposition}{Proposition}[section]
\theoremstyle{plain}
\newtheorem{lemma}{Lemma}[section]
\theoremstyle{plain}
\newtheorem{theorem}{Theorem}[section]
\theoremstyle{plain}

\theoremstyle{definition}

\newcommand{\Dkonv}{\stackrel{\mathcal{D}}{\longrightarrow}}

\begin{document}

\title{On a surprising behavior of the likelihood ratio test in non-parametric mixture models}
\author{Yan Zhang, Stanislav Volgushev}
\institute{University of Toronto}

\begin{abstract}
We study the likelihood ratio test in general mixture models where the base density is parametric, the null is a known fixed mixing distribution, and the alternative is a general mixing distribution supported on a bounded parameter space. For Gaussian location mixtures and Poisson mixtures, we show a surprising result: the non-parametric likelihood ratio test statistic converges to a tight limit if and only if the null distribution is a finite mixture, and diverges to infinity otherwise. We further demonstrate that the likelihood ratio test diverges for a fairly general class of distributions when the null mixing distribution is not finitely discrete. 
\end{abstract}

\maketitle

\section{Introduction}

Likelihood ratio tests (LRT) are among the most classical and widely used tools in statistical inference. The properties of such test are well understood in regular parametric models, and a vast literature exists on likelihood ratio tests in irregular models where the classical $\chi^2$ asymptotics can fail \citep{brazzale2024likelihood}. An important class of such irregular models that has been widely used in practice and has thus attracted substantial attention from the theoretical community are mixture models, see \cite{chen2023statistical} for a recent textbook treatment and \cite{titterington1985statistical,mclachlan2000finite,everitt2013finite} for classical textbooks. 

To formally introduce this model class, let $\Theta \subseteq \mathbb{R}^d$ denote the parameter space, and consider a family of densities $\{p_\theta : \theta \in \Theta\}$ on $\mathbb{R}^p$, each defined with respect to a common reference measure $\mu$ on $\mathbb{R}^p$. Denote by $\cG$ the set of all probability distributions supported on $\Theta$. For $g \in \cG$, let
\[
f_g(x) \coloneqq \int_\Theta p_\theta(x) \dif g(\theta). 
\]
The distributions $g$ are often called \textit{mixing distributions}. 

If $g$ is discrete with a finite, known number of components, say $K$, such models are known as \textit{finite mixtures} or \textit{$K$-component mixtures}. Classical regularity conditions fail in finite mixtures, and the properties of LRT in mixture models remained an enigma for a long time \citep{lindsay1995mixture}. Even the problem of LRT asymptotics for testing one versus two components in Gaussian location mixtures was only recently resolved in \cite{dacunha1997testing, dacunha1999testing, chen2001likelihood}. 

Early contributions to LRT properties for mixtures are \cite{GhoshSen1985, hartigan1985failure}, with later work by \cite{BickelChernoff1993, dacunha1997testing, dacunha1999testing, chen2001likelihood, ciuperca2002likelihood, liu2003testing, liu2004asymptotics, azais2006asymptotic} among many others; see \cite{chen2023statistical} and \cite{brazzale2024likelihood} for additional references. For the specific case of Gaussian location mixtures, the findings in the above literature can be summarized as follows: the LRT for one versus two components diverges to infinity when no restrictions are placed on the locations of the mass points of $g$ under the alternative \citep{hartigan1985failure, BickelChernoff1993, liu2004asymptotics, azais2006asymptotic} and converges to the supremum of a certain process when $g$ is restricted to have compact support \citep{dacunha1997testing, chen2001likelihood}. 

The non-parametric likelihood ratio test, where no restrictions are placed on $\cG$, is even more challenging to analyze and remains poorly understood. The most basic form of such a test is
\begin{equation}\label{eq:defLRT}
L_n(\cG,g_0) \coloneqq \sup_{g \in \cG} \ell_n(f_{g}) - \ell_n(f_{g_0}),
\end{equation}
where, for an i.i.d. sample $X_1,\dots, X_n$ from $f_{g_0}$, $\ell_n(f) \coloneqq \sum_{i=1}^n \log f(X_i)$ denotes the log-likelihood function and $g_0$ is a known, fixed distribution.

The findings in \cite{hartigan1985failure} imply that $L_n(\cG,g_0)$ diverges to infinity when $\Theta = \bR$, $p_\theta(\cdot) = \phi(\cdot-\theta)$ with $\phi$ denotes the standard normal pdf, and $g_0$ is a point mass at zero. The speed of divergence and additional details were subsequently studied in \cite{jiang2016generalized, jiang2019rate} among others. 

Similarly to the finite mixture case, the divergence described above is due to the fact that $\bR$ is not compact. Assuming that the parameter space $\Theta$ is bounded was shown to avoid issues with diverging LRT for Gaussian and Poisson mixtures, even for non-parametric likelihood ratio tests. Several authors provide abstract results on expansions and the convergence of likelihood ratio tests under high-level Donsker type conditions on certain function classes, which can incorporate non-parametric alternatives. This includes the work by \cite{gassiat2002likelihood, liu2003asymptotics, azais2009likelihood}. However, those conditions are abstract and verifying them even for very simple null models is very challenging. 

\cite{azais2009likelihood} verify those high level conditions in several non-parametric mixture models. 
They prove that $L_n(\mathcal{G}, g_0)$ converges in distribution in Hartigan's setting---i.e., when $g_0$ has a single point mass---provided that $\Theta$ is restricted to a compact interval. 
\cite{azais2009likelihood} provide an explicit expansion for $L_n(\cG,g_0)$ and show that the limit is given by the square of the supremum of the positive part of a certain Gaussian process. They prove similar results for Poisson mixtures, still assuming that $g_0$ is a degenerate distribution with a single point mass, and also study Binomial mixtures (in this case, $g_0$ is allowed to be more general). As key application of their results, \cite{azais2009likelihood} drive the asymptotic distribution of test for homogeneity in Gaussian, Poisson and Binomial mixtures where the null is that the sample is generated from $f_{g_0}$ with $g_0$ corresponding to a degenerate point mass at an unknown location while the alternative is that $g_0$ is a general distribution supported on a known, bounded interval.

Given the existing literature, it seems natural to conjecture that the likelihood ratio test in Gaussian and Poisson mixtures will converge in distribution even if the null is not a point mass, as long and we restrict the parameter space to be bounded. However, to the best of our knowledge, no results on the asymptotic behavior on the non-parametric LRT exist in Gaussian location mixtures or Poisson mixtures beyond what was proved in \cite{azais2009likelihood}. The proof technique in \cite{azais2009likelihood} uses the point mass structure of the null very explicitly, and does not extend beyond this particular case. 

The main finding in our paper is that the natural conjecture above is wrong. Specifically, we prove that the non-parametric LRT in Gaussian location mixtures and Poisson mixtures with bounded parameter space  converges if and only if $g_0$ is finitely discrete, and diverges to infinity otherwise. Intuitively, this is because finitely discrete $g_0$ are in a sense extremal points in the space of distributions and can be approached from fewer directions under the alternative than general $g_0$. 

\section{Main results} \label{sec:main}

Before presenting our main results, we introduce some additional notation. Throughout, we will use $\bN$ to denote the set of non-negative integers including zero. We will also write $[d]$ for $\{1,\dots,d\}$. For the class of mixing distributions $\cG$ as defined in the introduction, we will often write $\cF \coloneqq \{f_g: g \in \cG\}$ for the class of the resulting marginal distributions. For later ease of reference, we also formally define the Gaussian location mixture and Poisson mixture model as follows.
\begin{enumerate} \renewcommand{\theenumi}{(GM)}
\renewcommand{\labelenumi}{\theenumi}
\item \label{(GM)} $\Theta \subseteq \bR$ is bounded. We have $p_\theta(x) = \phi(x-\theta), x\in \bR$ where $\phi$ denotes the standard normal density and the base measure $\mu$ is Lebesgue measure on $\bR$.
\renewcommand{\theenumi}{(PM)}
\renewcommand{\labelenumi}{\theenumi}
\item \label{(PM)} $\Theta \subseteq (0,\infty)$ is bounded. We have $p_\theta(k) = \tfrac{\theta^ke^{-\theta}}{k!}, k \in \bN$ and the base measure $\mu$ is counting measure on $\bN$.
\end{enumerate}

Both \ref{(PM)} and \ref{(GM)} are important models in practice and have been studied extensively. For such models, we obtain a complete characterization of the behavior of $L_n(\cG,g_0)$ in terms of convergence/divergence.

\begin{theorem}\label{thm:MAIN}
Assume either \ref{(GM)} or \ref{(PM)}. 
\begin{enumerate}
\item If $g_0$ is discrete with a finite number of mass points, 
\[
L_n(\cG,g_0) \Dkonv \frac{1}{2}\sup_{s \in \cS} [(\bG(s))_+]^2, 
\]
where $x_+\coloneqq\max\{x, 0\}$ and $\bG(s)$ is a centered Gaussian process on $\cS$ with covariance structure 
\[
\bE[\bG(s_1)\bG(s_2)] = \bE[s_1(X)s_2(X)], \quad X \sim f_{g_0}, 
\]
and the score set $\cS$ is defined in~\eqref{eq:defS} below.
\item If $g_0$ is not finitely discrete\footnote{i.e. it is not a discrete distribution with a finite number of mass points}, $L_n(\cG,g_0)$ diverges to infinity in probability.
\end{enumerate}
\end{theorem}

Theorem~\ref{thm:MAIN} illustrates the core of our findings. It is a consequence of two more general results which we discuss in later sections: a general result on divergence of the LRT when $g_0$ is not finitely discrete (see Theorem~\ref{thm:diverge} in section~\ref{sec:div}) and convergence of the LRT in certain multivariate Gaussian/Poisson mixtures (see Theorem~\ref{thm:converge1} in section~\ref{sec:conv}). 

The \textit{score set} $\cS$ that is used to index the Gaussian process $\bG$ plays a key role in the asymptotic properties of the LRT. Intuitively, it can be thought of as characterizing all possible directions from which the null can be approached. We refer the interested reader to the discussions in \cite{gassiat2000likelihood, liu2003asymptotics, azais2009likelihood} and in section~\ref{sec:gen} for additional details. 

To describe the set $\cS$ more explicitly and provide some intuition for the reason behind the divergence result, we introduce some additional notation. For two densities $f,f_0$ with respect to a base measure $\mu$, the (square-rooted) chi-square divergence is defined as 
\begin{equation}\label{eq:defchi}
\chi(f,f_0) \coloneqq \Big\| \frac{f}{f_0} - 1 \Big\|_{L_2(f_0\dif \mu)}.
\end{equation}
A convenient way to parametrize the class of distributions $\cG$ is given by their moment sequences, this approach was also taken in~\cite{azais2009likelihood}. Specifically, fix a value $\theta_0$ and define \begin{equation}\label{def:MOM}
m_{k,g}\coloneqq\int_\Theta (\theta-\theta_0)^k\dif g(\theta), \quad k \geq 0. 
\end{equation}
When $g_0$ is finitely discrete, we will take $\theta_0$ to be a support point of $g_0$. Each $g \in \cG$ is uniquely determined by $\theta_0$ and $\{m_{k,g}\}_{k \in \bN}$ as guaranteed by the uniqueness property of the Hausdorff moment problem. Next we define orthogonal polynomials associated to $p_{\theta_0}$, as $q_0 \equiv 1$,
\begin{equation}\label{def:POL}
q_k(x)\coloneqq \frac{\partial^k}{\partial \theta^k}\frac{p_\theta(x)}{p_{\theta_0}(x)}\Bigg|_{\theta=\theta_0}, \quad k \geq 1. 
\end{equation}
When $p_\theta$ is Gaussian, $q_k$ are scaled versions of the Hermite polynomials, while for $p_\theta$ Poisson we obtain the (scaled) Poisson-Charlier polynomials \citep{morris1982natural}. In both cases, the sequence $\{q_k\}_{k \in\bN}$ is orthogonal in $L_2(p_{\theta_0}\dif\mu)$, i.e. $\int q_k(x)q_{k'}(x) p_{\theta_0}(x)\dif\mu(s) = 0$ for any $k \neq k'$. The set $\cS$ takes the form
\begin{equation}\label{eq:defS}
\cS \coloneqq \Big\{ s(\cdot) = \sqrt{\tfrac{p_{\theta_0}(\cdot)}{f_{g_0}(\cdot)}}\Big(\sum_{k=1}^\infty\tfrac{m_{k,g}-m_{k,g_0}}{k!\chi(f_g,f_{g_0})}q_k(\cdot)\sqrt{\tfrac{p_{\theta_0}(\cdot)}{f_{g_0}(\cdot)}}\Big):g\in\cG\backslash g_0\Big\}. 
\end{equation}

\begin{remark} A similar result was obtained in \cite{azais2009likelihood} in the case when $g_0$ is a point mass at $\theta_0$. In that case the ratio $\frac{p_{\theta_0}(\cdot)}{f_{g_0}(\cdot)}$ disappears and the score set contains weighted sums of orthogonal polynomials. \cite{azais2009likelihood} use this very explicitly, and extending their results beyond the case of degenerate $g_0$ requires a different approach. One of the key steps in this approach is discussed below.
\end{remark}

Note that by definition the functions $q_k(\cdot)\sqrt{p_{\theta_0}(\cdot)/f_{g_0}(\cdot)}$ are orthogonal in $L_2(f_{g_0}\dif\mu)$, so the size of the score set $\cS$ is determined by the sequences $\left\{(m_{k,g}-m_{k,g_0})/(k!\chi(f_g,f_{g_0}))\right\}_{k \in\bN}$. If $g_0$ is finitely discrete, we prove that the resulting class of sequences is not too rich. The key tool in showing this result is to realize that for discrete $g_0$ with at most $J$ components, higher-order moment differences $m_{k,g}-m_{k,g_0}$ can be controlled in terms of the first $2J$ moment differences. More formally, we have the following result which is of independent interest. The proof is given in the supplement. A similar result bounding higher-order moment differences using lower-order moment differences was established as Lemma 10 in \citet{wu2020optimal}, where both $g$ and $g_0$ are assumed to be finitely discrete.

\begin{lemma}\label{lem:MCL}
Assume that $\Theta \subseteq \bR$ is bounded. Let $g_0$ be a finite discrete distribution on $\Theta$ with $J$ support points. Fix an arbitrary $g\in\cG$ and define 
$$
\Delta_g\coloneqq\max_{k\in [2J]}|m_{k,g}-m_{k,g_0}|. 
$$
Then, for any $k>2J$ and $M\coloneqq\sup_{\theta\in\Theta}|\theta-\theta_0|$,
\begin{equation}\label{eq:MCL}
|m_{k,g}-m_{k,g_0}|\le k(M+1)^{2Jk}\Delta_g. 
\end{equation}
\end{lemma}

As \( J \) tends to infinity, the coefficients in \eqref{eq:MCL} diverge. Hence this result is only useful for finitely discrete \( g_0 \). In fact, for any \( g_0 \) that is not finitely discrete, one can construct a finite discrete distribution \( g \) that matches its first several moments but differs in higher-order moments---for example, via Gaussian quadrature methods. Thus, lower-order moment differences impose no constraints on higher-order moment differences. Consequently, such $g_0$ have much richer score sets. This can be seen as intuitive reason for the divergence of the LRT.

\subsection{Divergence of the likelihood ratio test}\label{sec:div}

In this section, we provide general results on the divergence of the non-parametric LRT~\eqref{eq:defLRT} when $g_0$ is not a finitely discrete distribution. We will only need to make the following mild assumptions, with the key point being identifiability in terms of the chi-square divergence of marginal distributions.

\begin{enumerate}\renewcommand{\theenumi}{(D)}
\renewcommand{\labelenumi}{\theenumi}
\item \label{(D)} The parameter set $\Theta \subseteq \bR^d$ is bounded, $g_0$ is not finitely discrete, and $\chi(f_g,f_{g_0}) = 0$ iff $g=g_0$. 
\end{enumerate}

\noindent
Assumption \ref{(D)} already implies divergence of the LRT.

\begin{theorem}\label{thm:diverge}
Under assumption \ref{(D)} we have in probability
\[
L_n(\cG,g_0) \to \infty.
\]
\end{theorem}

The divergence of a likelihood ratio test statistic in a mixture setting was first observed by \citet{hartigan1985failure} in the context of one-dimensional Gaussian location mixture models. The author considered mixtures of the form 
$$
f_{\theta,t}(x)\coloneqq(1-t)\phi(x)+t\phi(x-\theta), 
$$
where $\theta\in\bR$ and $t\in[0,1]$. For any $K\ge1$, they constructed a class of models
$$
\cF^{K}\coloneqq\big\{f_{\theta,t}:\theta\in\{\theta_1,\dots,\theta_K\},t\in[0,1]\big\}. 
$$
Assuming the data are generated from a standard normal distribution, they showed that when the values in $\{\theta_1,\dots,\theta_K\}$ are sufficiently well-separated, the likelihood ratio $\sup_{f \in \cF^K} \ell_n(f) - \ell_n(\phi)$ is bounded below by a random variable $\mathbb{L}_K$, which diverges to infinity as $K\to\infty$. This phenomenon fundamentally relies on the unboundedness of the parameter space for $\theta$. Proving divergence in our setting requires a different line of reasoning which explicitly takes into account the nature of the null distribution $g_0$.

The key to proving Theorem~\ref{thm:diverge} is to show that for any $K \geq 1$ there exist a subset $\cG^{\le K} \subseteq \cG$ such that $L_n(\cG^{\le K},g_0) \stackrel{\mathcal{D}}{\longrightarrow} \chi^2(K)$. Since $K$ can be taken arbitrarily large, divergence in probability follows. The idea for constructing such a class relies on considering a novel \textit{multiplicative} perturbations of the distribution $g_0$ by a weighted sum of orthogonal polynomials, say $\{q_k\}_{k\in\bN}$, that are associated with $g_0$\footnote{see Proposition~\ref{pro:ORT-multi} in the supplement for details}. Specifically, we set 
\begin{equation}\label{def:Gk}
\cG^{\le K}\coloneqq  \Big\{g\in\cG: \exists c \in \bR^K \text{ s.t. }
\dif g(\theta)=\Big(1+\sum_{k=1}^Kc_kq_k(\theta)\Big)\dif g_0(\theta)
\Big\}.    
\end{equation}
The corresponding score set is characterized in Lemma~\ref{lem:INF-multi}, where we show that it corresponds to all normalized linear combinations of a collection of $K$ linearly independent functions in $L_2(f_{g_0}\dif\mu)$. This form of score set results in a $\chi^2(K)$ limiting distribution.

When $g_0$ is finitely discrete, only a finite number of such orthogonal polynomials can be constructed, and divergence cannot be established by this approach. This proof strategy provides insights into the role of the null distribution for the divergence of the LRT: if $g_0$ is not finitely discrete, $g_0$ can be perturbed in ``too many'' directions, resulting in a very rich class of score functions which leads to divergence of the LRT. Finitely discrete $g_0$ are in a sense more extremal points of the space of distributions $\cG$, and can only be perturbed in certain directions. In specific models such as Gaussian location mixtures and Poisson mixtures, those directions are sufficiently ``few'' (but still infinitely many) to ensure tightness of the LRT.

\subsection{Convergence of the LRT in Gaussian and Poisson mixtures} \label{sec:conv}
 
In this section, we provide a general result on the convergence of the non-parametric likelihood ratio test~\eqref{eq:defLRT} for multivariate distributions with independent Poisson and Gaussian components.

\begin{enumerate}\renewcommand{\theenumi}{(C1)}
\renewcommand{\labelenumi}{\theenumi}
\item \label{(C1)} Fix some $b\in\{0,1,\dots,d\}$. The component densities have a product structure of the form
$$
p_\theta(x)=\prod_{l=1}^dp_{\theta_l}(x_l), 
$$
where $p_{\theta_l}$ follows \ref{(GM)} for $1 \leq l \leq b$ and \ref{(PM)} for $b+1 \leq l \leq d$. The set $\Theta \subseteq \bR^b \times (0,\infty)^{d-b}$ is bounded. The base measure is a corresponding product of Lebesgue measure and counting measure. The distribution $g_0$ is discrete with a finite number of mass points. 
\end{enumerate}

\begin{theorem}\label{thm:converge1}
Under assumption \ref{(C1)} we have 
\[
L_n(\cG,g_0) \Dkonv \frac{1}{2}\sup_{s \in \cS} [(\bG(s))_+]^2, 
\]
where $\bG(s)$ is a centered Gaussian process on $\cS$ with covariance structure given by 
\[
\bE[\bG(s_1)\bG(s_2)] = \bE[s_1(X)s_2(X)], \quad X \sim f_{g_0},
\]
and the set $\cS$ is defined in~\eqref{eq:defS-multi} in the supplement.
\end{theorem}

The score set $\cS$ has a similar structure as in the univariate case in~\eqref{eq:defS}, but is more complicated notationally because it depends on moment tensors. A formal definition is provided in section~\ref{sec:multS} in the supplement. As in the univariate case, differences in higher-order moments can be bounded by differences in lower-order moments, effectively yielding a finite number of degrees of freedom in the neighborhood of $g_0$. A multivariate extension of Lemma~\ref{lem:MCL} is provided in the Supplement as Lemma~\ref{lem:MCL-multi}. 

As was noted in \cite{liu2003testing}, the LRT also converges in fair generality when the distribution that is mixed has a finite number of support points. One result along those lines is provided in section~\ref{sec:finsup} of the supplement.

\subsection{General theory of the likelihood ratio test for star-shaped models}\label{sec:gen}

In this section, we present a general result on the behavior of the LRT under a ``star-shaped'' assumption. It simplifies some prior works in this particular setting and is a core ingredient in the proofs for the results in the previous sections. 

Consider a class of densities $\cF$ with respect to a measure $\mu$ and assume i.i.d. observations $X_1,\dots,X_n$ from a true density $f_0\in\cF$. We work in a general framework imposing the following three assumptions on the pair $(\mathcal{F},f_0)$: 
\begin{enumerate} \renewcommand{\theenumi}{(A1)}
\renewcommand{\labelenumi}{\theenumi}
\item \label{(A1)} For any $f\in\cF$, the convex combination $(1-t)f_0+tf$ remains in $\cF$ for all $t\in[0,1]$. 
\renewcommand{\theenumi}{(A2)}
\renewcommand{\labelenumi}{\theenumi}
\item \label{(A2)} Recall the definition of $\chi$ in~\eqref{eq:defchi}. For all $f\in\cF\backslash f_0$, 
$
\chi(f,f_0) \in (0,\infty).
$
\renewcommand{\theenumi}{(A3)}
\renewcommand{\labelenumi}{\theenumi}
\item \label{(A3)} The score set 
\begin{equation}\label{eq:SCO}
\cS\coloneqq\left\{s_f:f\in\cF\backslash f_0\right\},\quad s_f\coloneqq\frac{\frac{f}{f_0}-1}{\chi(f,f_0)}, 
\end{equation}
is $f_0\dif \mu$-Donsker and has an $f_0\dif \mu$-square integrable envelope. 
\end{enumerate}

Similar assumptions were previously imposed by \citet{gassiat2002likelihood}, \citet{liu2003asymptotics} and \citet{azais2009likelihood} who studied the behavior of the likelihood ratio test under very general conditions. Compared to those works, our assumptions are stronger in that we require a certain star-shaped structure of $\cF$ in \ref{(A1)}. This assumption is satisfied in mixture models with non-parametric mixing distributions, but fails in many other examples such as finite mixtures. \ref{(A1)} is thus tailored to our specific setting.   

The following theorem presents the main result of this section, establishing that the asymptotic behavior of the likelihood ratio test (LRT) statistic is fully characterized by a Gaussian process indexed by the score set. This result closely resembles Theorem 3.1 in \cite{liu2003asymptotics}. However, by assuming \ref{(A1)} we can remove the requirements of completeness and continuous sample paths that were imposed in the latter reference. Theorem~\ref{thm:ASY} is essentially contained in the proofs of \citet{gassiat2002likelihood} and \citet{azais2009likelihood}. However, in there it is not stated in the precise form we need, so we state it here with simpler notation and in the generality in which we will apply it subsequently. 

\begin{theorem}\label{thm:ASY}
Under \ref{(A1)}, \ref{(A2)} and \ref{(A3)}, it holds that
\begin{equation}\label{eq:ASY}
\sup_{f\in\cF}\ell_n(f)-\ell_n(f_0)=\frac{1}{2}\sup_{s\in\cS}[(\bG_n(s))_+]^2+o_{\bP}(1), 
\end{equation}
where $\bG_n$ denotes the empirical process
$$
\bG_n(s)\coloneqq\sqrt{n}\Big(\frac{1}{n}\sum_{i=1}^ns(X_i) - \bE[s(X_1)]
\Big). 
$$
\end{theorem}

For completeness we provide a full proof in the Supplement, section~\ref{sec:proofASY}. Similarly to the work of \citet{gassiat2002likelihood, liu2003asymptotics, azais2009likelihood}, the proof proceeds by showing that LHS $\geq$ RHS and LHS $\leq$ RHS. The proof of LHS $\leq$ RHS essentially follows the arguments in \cite{gassiat2002likelihood, azais2009likelihood}. The proof of LHS $\geq$ RHS is new and different from arguments in the existing literature. It utilizes \ref{(A1)} and allows us to avoid complicated discussions around differentiability in quadratic mean or similar arguments.

\section{Conclusions and discussion.}

This work establishes that, under non-parametric mixture models with Gaussian or Poisson components, the behavior of the likelihood ratio test (LRT) is governed by the structure of the null mixing distribution $g_0$. When $g_0$ is finitely discrete, the LRT converges, exhibiting an effectively finite-dimensional behavior despite the non-parametric model class. In contrast, when $g_0$ has infinitely many support points, the LRT diverges. This divergence is based on a new and general divergence mechanism beyond the non-compactness identified by \citet{hartigan1985failure}. In contrast to classical settings, our results reveal that convergence or divergence of the LRT is determined not only by the alternative but also by the particular form of the null hypothesis. Those results substantially advance our fundamental understanding of likelihood ratio statistics in non-parametric mixture models and will be useful for future methodological developments. 

Our proof method does not yields a rate of divergence when the LRT does diverge. Simulations in Gaussian location mixtures suggest a poly-logarithmic rate, which would be in line with the rates observed for unbounded parameter spaces \citep{jiang2016generalized,jiang2019rate}. However, the mechanisms underlying the divergence in those cases and in what we establish are different since the divergence we observe hinges on the specific form of $g_0$. Further investigations of this issue would be of interest but are beyond the scope of this paper. 

\bibliographystyle{apalike}
\bibliography{ref}

\section{Supplement} \label{sec:proof}

\subsection{Convergence of the LRT for distribution with a finite number of support points.} \label{sec:finsup}

Here, we briefly discuss the case where all $p_\theta$ can only take values in $\{1,\dots,K\}$ for a finite $K$. This includes Binomial and multinomial mixtures, where mixing is over success probabilities, as important special case. 

\begin{enumerate}\renewcommand{\theenumi}{(C2)}
\renewcommand{\labelenumi}{\theenumi}
\item \label{(C2)} There exists a $K \in \bN$ such that the densities $p_\theta$ are with respect to counting measure on $\{1,\dots,K\}$. The set $\Theta \subseteq \bR^d$ is arbitrary. The density $f_{g_0}$ is fully supported on $\{1,\dots,K\}$.
\end{enumerate}

A similar result was established as Theorem 3.2 in \citet{liu2003asymptotics} for more general discrete models. In contrast, we demonstrate that under the non-parametric setting, the requirements of completeness and continuous sample paths can be removed, and the LRT always converges.

\begin{theorem}\label{thm:converge2}
Under assumption \ref{(C2)}, $L_n(\cG,g_0)$ converges to a tight limit. 
\end{theorem}

Compared to assumption \ref{(C1)}, models satisfying condition \ref{(C2)} admit a more explicit upper bound on the limiting distribution. Specifically, the largest such model is the family of all discrete distributions supported on $\{1,\dots,K\}$. The likelihood ratio test statistic for this model converges to a chi-square distribution with $K-1$ degrees of freedom, $\chi^2(K-1)$, which serves as an upper bound for the limiting distribution of $L_n(\cG, g_0)$. 

\subsection{Details on the score set in Theorem~\ref{thm:converge1}} \label{sec:multS}
We now describe the score set $\cS$. Fix a support point $\theta_0$ of $g_0$. The characterization of $\cS$ relies on the set of moment tensors $\{m_{k,g}\}_{k\in\bN}$ which are a generalization of the univariate centered moments in~\eqref{def:MOM}. Specifically, for $\theta \in \bR^d$, the tensor $\theta^{\otimes k} \in (\bR^d)^{\otimes k}$ is a $k$-way array with entries $(\theta^{\otimes k})_{i_1,\dots,i_k} = \prod_{j=1}^k \theta_{i_j}$. With this notation, $m_{k,g}$ is defined as
$$
m_{k,g}\coloneqq\int_\Theta(\theta-\theta_0)^{\otimes k}\dif g(\theta). 
$$
Each $g\in\cG$ is uniquely determined by $\theta_0$ and $\{m_{k,g}\}_{k\in\bN}$ as guaranteed by the uniqueness property of the Hausdorff moment problem. 

Next we define the orthogonal polynomials associated to $p_{\theta_0}$, 
\begin{equation}\label{eq:polynomial}
q_\alpha(x)\coloneqq\frac{\partial^\alpha}{\partial\theta^\alpha}\frac{p_\theta(x)}{p_{\theta_0}(x)}\Bigg|_{\theta=\theta_0}\coloneqq\prod_{l=1}^d\frac{\partial^{\alpha_l}}{\partial\theta_l^{\alpha_l}}\frac{p_{\theta_l}(x_l)}{p_{\theta_{0,l}}(x_l)}\Bigg|_{\theta_l=\theta_{0,l}},\quad\alpha\in\bN^d. 
\end{equation}
Here, each $q_\alpha(x)$ is a product of Hermite polynomials and Poisson-Charlier polynomials and they are orthogonal in $L_2(p_{\theta_0}\dif\mu)$. With this notation, the set $\cS$ takes the form
\begin{equation}\label{eq:defS-multi}
\cS\coloneqq\left\{s(\cdot)=\sqrt{\frac{p_{\theta_0}(\cdot)}{f_{g_0}(\cdot)}}\left(\sum_{k=1}^\infty\sum_{|\alpha|=k}\frac{m_{\alpha,g}-m_{\alpha,g_0}}{\alpha!\chi(f_g,f_{g_0})}q_\alpha(\cdot)\sqrt{\frac{p_{\theta_0}(\cdot)}{f_{g_0}(\cdot)}}\right):g\in\cG\backslash g_0\right\}. 
\end{equation}
Here, for a multi-index $\alpha\in\bN^d$ and $\theta \in \bR^d$, $\theta^\alpha \coloneqq \prod_{i=1}^d \theta_i^{\alpha_i}$, $\alpha!\coloneqq\prod_{l=1}^d\alpha_l!$, $|\alpha|\coloneqq\sum_{i=1}^d\alpha_l$ and 
$$
m_{\alpha,g}\coloneqq\int(\theta-\theta_0)^\alpha\dif g(\theta). 
$$
Note that $m_{\alpha,g} \in \bR$ is an entry of the moment tensor $m_{|\alpha|,g} \in (\bR^d)^{\otimes |\alpha|}$.

\subsection{Proofs of main results}

\begin{proof}[Proof of Theorem~\ref{thm:MAIN}] 
The statement of part 1 and the expression for the score set follows directly from Theorem~\ref{thm:converge1}, setting $d=1$ and $b=0$ to obtain the Poisson case and $b=1$ for the Gaussian case. The statement of part 2 follows from Theorem~\ref{thm:diverge} upon noting that condition \ref{(D)} holds in the Gaussian case by elementary properties of characteristic functions combined with the fact that $f_g$ is the density of $Y+\eps$ where $Y \sim g$ and $\eps\sim N(0,1)$ independent of $Y$. In the Poisson case, condition \ref{(D)} follows because the probability-generating function of a Poisson mixture is the Laplace transform of the mixing distribution, and uniqueness of the Laplace transform ensures the result.
\end{proof}

\begin{proof}[Proof of Lemma~\ref{lem:MCL}] 
Here we prove a slightly stronger bound
\begin{equation}\label{eq:MCL-tmp}
|m_{k,g}-m_{k,g_0}|\le(k-2J)(M+1)^{2J(k-2J)+1}\Delta_g. 
\end{equation}
Let $\theta_1,\dots,\theta_J$ be the support points of $g_0$. We will repeatedly use the following representation
\begin{equation}\label{eq:help1mom}
\prod_{j=1}^J(\theta-\theta_j)^2 = \prod_{j=1}^{2J}(\theta-\theta_0+\theta_0-\kappa_j) = \sum_{j=0}^{2J} (\theta - \theta_0)^{2J-j}\sum_{s \subseteq [2J]; |s| = j} \prod_{i \in s} (\theta_0 - \kappa_i), 
\end{equation}
where $\kappa_{2j} = \kappa_{2j-1} = \theta_j$, along with the bound
\begin{equation}\label{eq:help2mom}
\Big| \sum_{s \subseteq [2J]: |s| = j} \prod_{i \in s} (\theta_0 - \kappa_i) \Big| \leq \binom{2J}{j} M^{j}.
\end{equation}
We begin by noting
$$
\begin{aligned}
\left|\int(\theta-\theta_0)^{k-2J}\prod_{j=1}^J(\theta-\theta_j)^2\dif(g-g_0)(\theta)\right|
&=
\left|\int(\theta-\theta_0)^{k-2J}\prod_{j=1}^J(\theta-\theta_j)^2\dif g(\theta)\right|
\\
&\le M^{k-2J}\int\prod_{j=1}^J(\theta-\theta_j)^2\dif g(\theta)
\\
&=M^{k-2J}\left|\int\prod_{j=1}^J(\theta-\theta_j)^2\dif(g-g_0)(\theta)\right|
\\
&\le M^{k-2J}\sum_{j=0}^{2J}\binom{2J}{j}M^{j}\left|\int(\theta-\theta_0)^{2J-j}\dif(g-g_0)(\theta)\right|
\\
&\le (M+1)^{2J}M^{k-2J}\Delta_g, 
\end{aligned}
$$
where we used \eqref{eq:help1mom}, \eqref{eq:help2mom} in the second inequality and the identity
$$
\sum_{j=0}^{2J}\binom{2J}{j}M^{j}=(M+1)^{2J}, 
$$
in the last inequality. 
To proceed, recall $k > 2J$ and observe that
$$
\begin{aligned}
|m_{k,g}-m_{k,g_0}|
&= \left|\int(\theta-\theta_0)^k\dif(g-g_0)(\theta)\right|
\\
&\le (M+1)^{2J}M^{k-2J}\Delta_g+\left|\int\Big((\theta-\theta_0)^k-(\theta-\theta_0)^{k-2J}\prod_{j=1}^J(\theta-\theta_j)^2\Big)\dif(g-g_0)(\theta)\right|
\\
&\le (M+1)^{2J}M^{k-2J}\Delta_g+\sum_{j=1}^{2J}\binom{2J}{j}M^{j}\left|\int(\theta-\theta_0)^{k-j}\dif(g-g_0)(\theta)\right|
\\
&\le (M+1)^{2J}M^{k-2J}\Delta_g+(M+1)^{2J}\sup_{j\in[2J]}|m_{k-j,g}-m_{k-j,g_0}|, 
\end{aligned}
$$
where we used \eqref{eq:help1mom} and \eqref{eq:help2mom} in the second inequality. Now, we prove \eqref{eq:MCL-tmp} by induction. When $k=2J+1$, from the above formula, 
$$
\begin{aligned}
|m_{2J+1,g}-m_{2J+1,g_0}|&\le(M+1)^{2J}M\Delta_g+(M+1)^{2J}\Delta_g\\
&=(M+1)^{2J+1}\Delta_g. 
\end{aligned}
$$
Suppose \eqref{eq:MCL-tmp} holds for $k-1>2J$. Then, 
$$
\begin{aligned}
|m_{k,g}-m_{k,g_0}|&\le(M+1)^{2J}M^{k-2J}\Delta_g+(k-1-2J)(M+1)^{2J(k-2J)+1}\Delta_g\\
&\le(k-2J)(M+1)^{2J(k-2J)+1}\Delta_g, 
\end{aligned}
$$
where we used 
\[
2J(k-2J)+1 - k = (2J-1)k -4J^2+1 \geq (2J-1)(2J+1) -4J^2+1 = 0, 
\]
in the last inequality. This completes the proof. 
\end{proof}

\subsection{Proof of Theorem~\ref{thm:diverge}}

The analysis crucially relies on a set of orthogonal polynomials associated with $g_0$. 

\begin{proposition}\label{pro:ORT-multi}
Assume that $g_0$ is not finitely discrete and is supported on a compact set. Then there exists a sequence of polynomials $\{q_k\}_{k\in\bN}$ on $\Theta$ satisfying $q_0(\theta)\equiv1$ and, for any $k,k'\in\bN$, $\int_\Theta q_{k}(\theta)q_{k'}(\theta)\dif g_0(\theta) = \mathbf{1}_{\{k=k'\}}$.
\end{proposition}

\begin{proof}
Since $g_0$ is not finitely discrete, there must exist an index $l \in [d]$ such that  the $l$th marginal of $g_0$ is not finite discrete. For this $l$, the elements in $\{1,\theta_l,\theta_l^2,\dots\}$ are linearly independent in $L_2(g_0)$. The desired polynomials can now be constructed through the Gram-Schmidt process. 
\end{proof}

These polynomials provide a parameterization of a nested sequence of subsets of $\cG$. For a nonnegative integer $K$, we define the $K$-order sub-model by
$$
\cF^{\le K}\coloneqq\left\{f_g:g\in\cG^{\le K}\right\},\, \quad \cG^{\le K}\coloneqq  \left\{g\in\cG: \exists c \in \bR^K \text{ s.t. }
\dif g(\theta)=\left(1+\sum_{k=1}^Kc_kq_k(\theta)\right)\dif g_0(\theta)
\right\}. 
$$
The corresponding score set has an explicit expression. 

\begin{lemma}\label{lem:INF-multi}
Assume \ref{(D)}. Fix $K \geq 1$. The score set corresponding to the $K$-order sub-model $\cG^{\le K}$ takes the following form
\begin{equation}\label{eq:SCO-multi}
\cS^{\le K}\coloneqq\left\{s_{f_g}:g\in\cG^{\le K}\backslash g_0\right\}=\left\{\frac{\sum_{k=1}^Kc_kh_k}{\|\sum_{k=1}^Kc_kh_k\|_{L_2(f_{g_0}\dif\mu)}}:c\in\bS^{K-1}\right\}, 
\end{equation}
where $\bS^{K-1}$ denotes the surface of $K$-dimensional unit ball in $\bR^K$ and the functions
$$
h_k(x)\coloneqq\frac{\int p_\theta(x) q_k(\theta)\dif g_0(\theta)}{f_{g_0}(x)}\mathbf{1}_{\{f_{g_0}(x) > 0\}} , 
$$
are linearly independent elements of $L_2(f_{g_0}\dif\mu)$.
\end{lemma}

\begin{proof}[Proof of Lemma~\ref{lem:INF-multi}]
To proceed, we establish two key facts. First, for $k\in[K]$, $h_k$ is $f_{g_0}\dif\mu$-square integrable since $|h_k(x)|\le\sup_{\theta\in\Theta}|q_k(\theta)|<\infty$. Therefore, for $\dif g_c=(1+\sum_{k=1}^Kc_kq_k)\dif g_0\in\cG^{\le K}$, 
\begin{equation}\label{eq:CHI-multi}
\chi(f_{g_c},f_{g_0})=\left\|\sum_{k=1}^Kc_kh_k\right\|_{L_2(f_{g_0}\dif\mu)}<\infty.
\end{equation}
We also note that for such $g_c$ with $c \neq 0$, 
\[
\int_\Theta q_k(\theta) \dif g_c(\theta) = c_k, 
\]
while $\int_\Theta q_k(\theta) \dif g_0(\theta) = 0$ for $k \in [K]$. Therefore $g_c \neq g_0$ and by assumption \ref{(D)}, 
\begin{equation}\label{eq:NOM-multi}
\left\|\sum_{k=1}^K c_k h_k\right\|_{L_2(f_{g_0}\dif\mu)} = \chi(f_{g_{c}},f_{g_0})>0. 
\end{equation}
Linear independence of the $h_k$ follows. Furthermore, 
\[
\cS^{\le K} \subseteq \left\{\frac{\sum_{k=1}^Kc_kh_k}{\|\sum_{k=1}^Kc_kh_k\|_{L_2(f_{g_0}\dif\mu)}}:c\in\bS^{K-1}\right\},
\]
and it remains to establish the inclusion in the other direction.

To this end, fix an arbitrary $c \in \bS^{K-1}$ and let 
\[
C \coloneqq \sum_{k=1}^K\sup_{\theta\in\Theta} |q_k(\theta)|. 
\]
By the definition of $C$ we have $\sup_{\theta \in \Theta } \frac{1}{C}\sum_{k=1}^K|c_kq_k(\theta)| \leq 1$, and since $\int_\Theta q_k(\theta) dg_0(\theta) = 0$ for $k \in [K]$, it follows that $g_{c/C}$ is a probability measure on $\Theta$. Noting that the score of $f_{g_{c/C}}$ is 
\[
\frac{\sum_{k=1}^Kc_kh_k}{\|\sum_{k=1}^Kc_kh_k\|_{L_2(f_{g_0}\dif\mu)}},
\]
we obtain
\[
\left\{\frac{\sum_{k=1}^Kc_kh_k}{\|\sum_{k=1}^Kc_kh_k\|_{L_2(f_{g_0}\dif\mu)}}:c\in\bS^{K-1}\right\} \subseteq \cS^{\le K}. 
\]
This completes the proof. \end{proof}

We are now ready to state and prove the key result which will imply the statement on Theorem~\ref{thm:diverge}. 

\begin{theorem}\label{thm:INF-multi}
Assume \ref{(D)}. The pair $(\cF^{\le K},f_{g_0})$ satisfies \ref{(A1)}, \ref{(A2)} and \ref{(A3)} for any $K\ge1$. Moreover, for any $K \geq 1$,
$$
2\left(\sup_{f\in\cF^{\le K}}\ell_n(f)-\ell_n(f_{g_0})\right)\Dkonv\chi^2(K). 
$$
\end{theorem}

\begin{proof}[Proof of Theorem~\ref{thm:INF-multi}] 
For any $K\ge1$, the pair $(\cF^{\le K},f_{g_0})$ satisfies \ref{(A1)} and \ref{(A2)} by definition and the formula \eqref{eq:CHI-multi}. We now turn to verifying assumption \ref{(A3)} for the score set $\cS^{\le K}$, as defined in \eqref{eq:SCO-multi}. Using the formula \eqref{eq:NOM-multi} and the compactness of $\bS^{K-1}$, 
\begin{equation}\label{eq:hlinindep-multi}
\inf_{c\in\bS^{K-1}}\left\|\sum_{k=1}^Kc_kh_k\right\|_{L_2(f_{g_0}\dif\mu)}>0. 
\end{equation}
Therefore, in light of \eqref{eq:SCO-multi}, $\cS^{\le K}$ is, after scaling, a subset of the convex hull of the finitely many $f_{g_0}\dif \mu$-square integrable functions $\{h_k, k=1,\dots,K\}$, ensuring that assumption \ref{(A3)} holds; see Theorem 2.10.3 in \cite{van1996weak}. 

By Theorem \ref{thm:ASY}, 
$$
2\left(\sup_{f\in\cF^{\le K}}\ell_n(f)-\ell_n(f_{g_0})\right)\to\sup_{s\in\cS^{\le K}}[(\bG(s))_+]^2, 
$$
in distribution, where $\bG$ is a centered Gaussian process indexed by $\cS^{\le K}$, with covariance function 
$$
\mathrm{Cov}(\bG(s_1),\bG(s_2))\coloneqq\int s_1(x)s_2(x) f_{g_0}(x) \dif \mu(x), 
$$
for any $s_1,s_2\in\cS^{\le K}$. This process admits an equivalent representation in terms of a standard Gaussian vector $Z\sim N(0,I_K)$ and a full rank covariance matrix $\Sigma \in \bR^{K\times K}$ with entries
$$
\Sigma_{k_1,k_2}=\int h_{k_1}(x)h_{k_2}(x) f_{g_0}(x) \dif \mu(x), \quad 1\le k_1,k_2\le K.
$$
To see that $\Sigma$ has full rank, note that if this was not the case there would exist an $a \in \bS^{K-1}$ such that 
\[
0 = a^\top \Sigma a = \sum_{i,j=1}^K \int a_i a_j h_i(x) h_j(x) f_{g_0}(x) \dif \mu(x) = \int \Big(\sum_{j=1}^K a_j h_j(x) \Big)^2 f_{g_0}(x) \dif \mu(x),
\]
a contradiction to \eqref{eq:hlinindep-multi}.

Specifically, by Lemma~\ref{lem:INF-multi}, for every score function $s \in\cS^{\le K} $ there exists a $c = c(s) \in \bS^{K-1}$ such that
$$
s(x) = s_c(x) \coloneqq \frac{\sum_{k=1}^Kc_kh_k}{\|\sum_{k=1}^Kc_kh_k\|_{L_2(f_{g_0}\dif\mu)}}.
$$
It is then straightforward to verify that the centered Gaussian process $\widetilde\bG$ defined through $\widetilde \bG(c) \coloneqq  \frac{c^\top\Sigma^{1/2}Z}{\sqrt{c^\top\Sigma c}}, c \in \bS^{K-1}$ satisfies
\[
\mathrm{Cov}(\widetilde\bG(c),\widetilde\bG(c')) = \mathrm{Cov}(\bG(s_c),\bG(s_{c'})).
\]
Finally, a direct calculation yields the chi-square limit: 
$$
\sup_{s\in\cS^{\le K}}[(\bG(s))_+]^2 \stackrel{\mathcal{D}}{=} \sup_{c\in\bS^{K-1}}[(\widetilde\bG(c))_+]^2 = \left(\sup_{c\in\bS^{K-1}}\frac{c^\top\Sigma^{1/2}Z}{\sqrt{c^\top\Sigma c}}\right)^2=\left(\sup_{c\in\bS^{K-1}}c^\top Z\right)^2=Z^\top Z, 
$$
where we used that, one of $\widetilde\bG(c), \widetilde\bG(-c)$ is always non-negative and since $\Sigma$ has full rank, 
\[
\Big\{ \frac{\Sigma^{1/2}c}{\sqrt{c^\top \Sigma c}}: c \in \bS^{K-1} \Big\} = \bS^{K-1}. 
\]
~ \end{proof}

\begin{proof}[Proof of Theorem~\ref{thm:diverge}]
Theorem~\ref{thm:diverge} is now a simple consequence of Theorem \ref{thm:INF-multi}. Indeed, for any $M > 0, K>1$ we have by the Portmanteau Theorem, 
\[
\liminf_{n \to \infty} \bP\Big(\sup_{f \in \cF} \ell_n(f) - \ell_n(f_{g_0}) > M \Big) \geq \lim_{n\to \infty} \bP\Big(\sup_{f \in \cF^{\le K}} \ell_n(f) - \ell_n(f_{g_0}) > M \Big) \ge \bP(\chi^2(K) > M).
\]
The probability of the right-hand side can be made arbitrarily close to one by selecting a sufficiently large $K$, and so for any $M > 0$, 
\[
\liminf_{n \to \infty} \bP\Big(\sup_{f \in \cF} \ell_n(f) - \ell_n(f_{g_0}) > M \Big) = 1.
\]
~ \end{proof}

\subsection{Proof of Theorem \ref{thm:converge1}}

For an order-$k$ tensor $T \in (\mathbb{R}^d)^{\otimes k}$, we write $\|T\|_\infty$ for its maximum absolute entry, $\|T\|_F$ for the Frobenius norm defined as the square root of the sum of squared entries, and define the spectral norm by  
\[
\|T\|_2 \coloneqq \sup_{c_1,\dots,c_k \in \mathbb{S}^{d-1}} \langle T, c_1 \otimes \dots \otimes c_k\rangle, 
\]
where for two tensors $T, T^\prime$, $\langle T,T^\prime\rangle$ denotes the inner product of their vectorized versions. These norms satisfy the inequalities
\[
\|T\|_\infty \;\le\; \|T\|_2 \;\le\; \|T\|_F \;\le\; d^{k/2}\|T\|_\infty.
\]
A tensor $T$ is called symmetric if
\[
T_{j_1,\dots,j_k} = T_{j_{\pi(1)},\dots,j_{\pi(k)}}
\quad \text{for all } j_1,\dots,j_k \in [d] \text{ and all permutations $\pi$ on }[k].
\]
In particular, the moment tensors introduced in section~\ref{sec:multS} are symmetric. For symmetric tensors, a classical result due to Banach \citep{banach1938homogene,friedland2018nuclear} gives the sharper characterization
\begin{equation}\label{eq:SPE}
\|T\|_2 = \sup_{c \in \mathbb{S}^{d-1}} \bigl|\langle T, c^{\otimes k}\rangle\bigr|.
\end{equation}
Throughout this section, we adopt the notations from Section~\ref{sec:multS} and additionally define
\[
M \coloneqq \sup_{\theta \in \Theta} \|\theta - \theta_0\|,
\]
assuming \(M < \infty\), where \(\|\cdot\|\) denotes the Euclidean norm on \(\mathbb{R}^d\).

By Theorem 4 and Corollary 1 in \citet{morris1982natural}, we have
\begin{proposition}\label{pro:polynomial}
Recall the definition of $q_\alpha$ in~\eqref{eq:polynomial}. We have for $\alpha,\alpha'\in\bN^d$: 
\begin{enumerate}[(i)]
\item $\int q_{\alpha}(x)q_{\alpha'}(x) p_{\theta_0}(x) \dif \mu(x) = a_\alpha \alpha!\mathbf{1}_{\{\alpha=\alpha'\}}$. 
\item $\int q_\alpha(x) p_{\theta_0}(x) \dif \mu(x) = a_\alpha(\theta-\theta_0)^\alpha$. 
\end{enumerate}
Here, 
$$
a_\alpha\coloneqq\prod_{l=1}^d\frac{1}{V(\theta_{0,l})^{\alpha_l}},\quad V(\theta_{0,l})\coloneqq\begin{cases}
1, & \text{if the $l$th marginal is Gaussian;}\\
\theta_{0,l}, & \text{if the $l$th marginal is Poisson.}
\end{cases}
$$
\end{proposition}

As noted by \citet{azais2009likelihood}, the relationship between the numerators of the score functions \eqref{eq:SCO} and the moments $\{m_{k,g}\}_{k\in\bN}$ can be derived via a Taylor series expansion. 

\begin{lemma}\label{lem:NUM-multi}
Assume \ref{(C1)}. For any $g\in\cG$, 
$$
\frac{f_g(x)}{p_{\theta_0}(x)}-1=\sum_{k=1}^\infty\sum_{|\alpha|=k}\frac{m_{\alpha,g}}{\alpha!}q_\alpha(x) 
$$
and the series converges absolutely for any $x \in \mathrm{supp}(f_{g_0})$.
\end{lemma}

\begin{proof} 
Under \ref{(C1)} the function $\theta \mapsto p_\theta(x)$ is a product of entire functions (each in a different argument), and thus the corresponding Taylor series 
$$
\frac{p_\theta(x)}{p_{\theta_0}(x)}-1=\sum_{k=1}^\infty\sum_{|\alpha|=k}\frac{(\theta-\theta_0)^\alpha}{\alpha!}q_\alpha(x) 
$$
converges absolutely everywhere by Theorem 1.2.5 and Corollary 2.3.7 from \citet{krantz2001function}. Furthermore, we have the bound
$$
\left|\frac{(\theta-\theta_0)^\alpha}{\alpha!}q_\alpha(x)\right|\le\frac{M^{|\alpha|}}{\alpha!}|q_\alpha(x)|. 
$$
The sum $\sum_{k=1}^\infty\sum_{|\alpha|=k} \tfrac{M^k}{\alpha!}|q_\alpha(x)|$ converges since the Taylor series converges absolutely at $\theta=\theta_0+M$. Applying Fubini's theorem to justify interchanging the sum and the integral, we obtain
$$
\begin{aligned}
\frac{f_g(x)}{p_{\theta_0}(x)}-1&=\int\left(\frac{p_\theta(x)}{p_{\theta_0}(x)}-1\right)\dif g(\theta)\\
&=\int\left(\sum_{k=1}^\infty\sum_{|\alpha|=k}\frac{(\theta-\theta_0)^\alpha}{\alpha!}q_\alpha(x)\right)\dif g(\theta)\\
&=\sum_{k=1}^\infty\sum_{|\alpha|=k}\frac{m_{\alpha,g}}{\alpha!}q_\alpha(x). 
\end{aligned}
$$
$~~$
\end{proof}

The next lemma compares the denominators of the score functions \eqref{eq:SCO} with the moment differences. A similar result was derived in Theorem 9 of \citet{bandeira2020optimal} for the case of multivariate Gaussian mixture models. 

\begin{lemma}\label{lem:DEN-multi}
Under \ref{(C1)}, for any $g\in\cG$, 
\begin{multline*}
\sup_{k\in\bN}\left(\min_{|\alpha|=k}\frac{a_\alpha^2}{\int q_\alpha^2(x)f_{g_0}(x)\dif\mu(x)}\right)\|m_{k,g}-m_{k,g_0}\|_\infty^2
\\
\le \chi^2(f_g,f_{g_0})
\le C_0\sum_{k=1}^\infty\left(\sup_{|\alpha|=k}a_\alpha\right)\frac{\|m_{k,g}-m_{k,g_0}\|_F^2}{k!},   
\end{multline*}
where
$$
C_0\coloneqq\sup_{x\in\mathrm{supp}(f_{g_0})}\frac{p_{\theta_0}(x)}{f_{g_0}(x)}<\infty. 
$$
\end{lemma}

\begin{proof} 
The fact that $C_0<\infty$ is guaranteed by the fact that $\theta_0$ is a support point of $g_0$. To derive the upper bound, observe that
$$
\begin{aligned}
\chi^2(f_g,f_{g_0}) &= \int\left(\frac{f_g(x)}{f_{g_0}(x)}-1\right)^2 f_{g_0}(x) \dif\mu(x)
\\
&=\int\left(\frac{f_g(x)-f_{g_0}(x)}{p_{\theta_0}(x)}\right)^2\frac{p_{\theta_0}(x)}{f_{g_0}(x)} p_{\theta_0}(x) \dif\mu(x)
\\
&\le C_0\int\left(\sum_{k=1}^\infty\sum_{|\alpha|=k}\frac{m_{\alpha,g}-m_{\alpha,g_0}}{\alpha!}q_\alpha(x)\right)^2 p_{\theta_0}(x) \dif\mu(x)
\\
&\le C_0\sum_{k=1}^\infty\sum_{|\alpha|=k} a_\alpha\frac{(m_{\alpha,g}-m_{\alpha,g_0})^2}{\alpha!}\\
&\le C_0\sum_{k=1}^\infty\left(\sup_{|\alpha|=k}a_\alpha\right)\frac{\|m_{k,g}-m_{k,g_0}\|_F^2}{k!}, 
\end{aligned}
$$
where the third step follows from Lemma \ref{lem:NUM-multi} and the definition of $C_0$. The fourth inequality is justified by property (i) of Proposition \ref{pro:polynomial}. Indeed, if the sum in the fourth line is infinite, the upper bound becomes trivial. If the sum in the fourth line is finite, then by Proposition \ref{pro:polynomial} (i) the sequence 
$$
\left\{\sum_{k=1}^K\sum_{|\alpha|=k}\frac{m_{\alpha,g}-m_{\alpha,g_0}}{\alpha!}q_\alpha\right\}_{K\in\bN}
$$
forms a Cauchy sequence in $L_2(p_{\theta_0}\dif \mu)$ and we have for any fixed $K$, 
\[
\Big\|\sum_{k=1}^K\sum_{|\alpha|=k}\frac{m_{\alpha,g}-m_{\alpha,g_0}}{\alpha!}q_\alpha \Big\|_{L_2(p_{\theta_0}\dif \mu)}^2 = \sum_{k=1}^K\sum_{|\alpha|=k} a_\alpha\frac{(m_{\alpha,g}-m_{\alpha,g_0})^2}{\alpha!}.
\]
In this case the inequality is in fact an equality. Finally, the last inequality is a consequence of the fact that
$$
\|m_{k,g}-m_{k,g_0}\|_F^2=\sum_{|\alpha|=k}\frac{k!}{\alpha!}(m_{\alpha,g}-m_{\alpha,g_0})^2. 
$$

To prove the lower bound, applying property (ii) in Proposition \ref{pro:polynomial} and the Cauchy–Schwarz inequality, we have
\begin{align*}
|a_\alpha| |m_{\alpha,g}-m_{\alpha,g_0}|
&= \left|\int q_\alpha(x) f_g(x) \dif\mu(x)-\int q_\alpha(x)f_{g_0}(x) \dif\mu(x)\right|
\\
&= \left|\int q_\alpha(x)\left(\frac{f_g(x)}{f_{g_0}(x)}-1\right) f_{g_0}(x) \dif\mu(x)\right|
\\
&\le\chi(f_g,f_{g_0})\sqrt{\int q_\alpha^2(x) f_{g_0}(x) \dif\mu(x)}, 
\end{align*}
for any $\alpha\in\bN^d$. 
\end{proof}

The moment comparison lemma given below plays a central role in the proof of the theorem and may also be of independent interest. 

\begin{lemma}\label{lem:MCL-multi}
Let $g_0$ be a finite discrete distribution with $J$ support points. Fix some $g\in\cG$ and define 
$$
\Delta_g\coloneqq\max_{k\in[2J]}\|m_{k,g}-m_{k,g_0}\|_2. 
$$
Then, for any $k>2J$, 
\begin{equation}\label{eq:MCL-multi}
\|m_{k,g}-m_{k,g_0}\|_2\le k(M+1)^{2Jk}\Delta_g. 
\end{equation}
\end{lemma}

\begin{proof}[Proof of Lemma~\ref{lem:MCL-multi}]
This lemma is a direct generalization of Lemma~\ref{lem:MCL}. For any $k>2J$, we have
$$
\begin{aligned}
\|m_{k,g}-m_{k,g_0}\|_2&=\sup_{c\in\bS^{d-1}}|\langle m_{k,g}-m_{k,g_0},c^{\otimes k}\rangle|\\
&=\sup_{c\in\mathbb{S}^{d-1}}\left|\int\langle\theta-\theta_0,c\rangle^k\dif(g-g_0)(\theta)\right|\\
&\le k(M+1)^{2Jk}\sup_{c\in\mathbb{S}^{d-1}}\max_{k\in[2J]}\left|\int\langle\theta-\theta_0,c\rangle^k\dif(g-g_0)(\theta)\right|\\
&= k(M+1)^{2Jk}\Delta_g. 
\end{aligned}. 
$$
Here, the first and last equality follows from the equality \eqref{eq:SPE}. The inequality follows by applying Lemma~\ref{lem:MCL} to the distributions of $\langle\theta,c\rangle$ for $\theta \sim g_0$ and $\theta \sim g$, respectively. Note that if $g_0$ has $J$ support points, then the corresponding distribution of $\langle\theta,c\rangle$ has no more than $J$ support points, and Lemma~\ref{lem:MCL} still evidently applies when the number of support points of $g_0$ is at most $J$. Also, by the definition of $M$, $|\langle\theta,c\rangle-\langle\theta_0,c\rangle| \leq M$ for $c \in \bS^{d-1}, \theta \in \Theta$. 
\end{proof}

\bigskip

\begin{proof}[Proof of Theorem~\ref{thm:converge1}] 
We will apply Theorem~\ref{thm:ASY} with $f_0 = f_{g_0}$, $\cF = \{f_g: g \in \cG\}$. Assumption \ref{(A1)} is satisfied by definition. Regarding assumption \ref{(A2)}, a simple computation shows that $\|m_{k,g} - m_{k,g_0}\|_F^2 \leq 4d^kM^{2k}$. Thus by Lemma~\ref{lem:DEN-multi} 
\[
\chi^2(f_g,f_{g_0}) \leq 4C_0\sum_{k=1}^\infty\left(\sup_{|\alpha|=k}a_\alpha\right)\frac{d^kM^{2k}}{k!}, 
\]
which is finite under \ref{(C1)} (recall the values for $a_\alpha$ in Proposition~\ref{pro:polynomial}). It remains to verify assumption \ref{(A3)}. 

By Lemma \ref{lem:NUM-multi}, for any $g\in\cG$, we have
$$
\begin{aligned}
\frac{f_g}{f_{g_0}}-1&=\frac{p_{\theta_0}}{f_{g_0}}\left(\frac{f_g-f_{g_0}}{p_{\theta_0}}\right)
\\
&=\frac{p_{\theta_0}}{f_{g_0}}\left(\sum_{k=1}^\infty\sum_{|\alpha|=k}\frac{m_{\alpha,g}-m_{\alpha,g_0}}{\alpha!}q_\alpha\right)
\\
&=\sqrt{\frac{p_{\theta_0}}{f_{g_0}}}\left(\sum_{k=1}^\infty\sum_{|\alpha|=k}\frac{m_{\alpha,g}-m_{\alpha,g_0}}{\alpha!}q_\alpha\sqrt{\frac{p_{\theta_0}}{f_{g_0}}}\right). 
\end{aligned}
$$
Hence, the score set can be written as
$$
\cS= \left\{\sqrt{\frac{p_{\theta_0}}{f_{g_0}}}\sum_{k=1}^\infty\sum_{|\alpha|=k} c_{\alpha,g}h_\alpha:g\in\cG\backslash g_0\right\}, 
$$
where 
$$
c_{\alpha,g}\coloneqq\sqrt{\frac{a_\alpha (k+1)^d}{\alpha!}}\frac{|\alpha|(m_{\alpha,g}-m_{\alpha,g_0})}{\chi(f_g,f_{g_0})},\, h_\alpha\coloneqq\frac{q_\alpha}{|\alpha|\sqrt{a_\alpha (k+1)^d\alpha!}}\sqrt{\frac{p_{\theta_0}}{f_{g_0}}}. 
$$

Because $\sqrt{p_{\theta_0}/f_{g_0}}$ is uniformly bounded by Lemma~\ref{lem:DEN-multi}, Example 2.10.10 of \citet{van1996weak} implies that it suffices to verify the Donsker property and to identify an appropriate envelope function for the family 
$$
\left\{\sum_{k=1}^\infty\sum_{|\alpha|=k} c_{\alpha,g}h_\alpha:g\in\cG\backslash g_0\right\}. 
$$
According to Theorem 2.13.2 \citep{van1996weak}, this family is $f_{g_0} \dif \mu$-Donsker as long as: 
\begin{enumerate}[(a)]
\item $\{h_\alpha\}_{\alpha\in\bN^d}$ is an orthogonal sequence in $L_2(f_{g_0} \dif \mu)$ with $\sum_{k=1}^\infty\sum_{|\alpha|=k}\|h_\alpha\|_{L_2(f_{g_0}\dif \mu)}^2<\infty$. 
\item For any $g\in\cG\backslash g_0$, $\sum_{k=1}^\infty\sum_{|\alpha|=k} c_{\alpha,g}h_\alpha$ converges pointwise, and $\sup_{g\in \cG}\sum_{k=1}^\infty\sum_{|\alpha|=k} c_{\alpha,g}^2 < \infty$. 
\end{enumerate}
Condition (a) follows directly from the definition and property (i) in Proposition \ref{pro:polynomial}. To verify (b), we use the lower bound in Lemma \ref{lem:DEN-multi}. It states that
$$
\begin{aligned}
\chi^2(f_g,f_{g_0})&\ge\max_{k\in[2J]}\left(\min_{|\alpha|=k}\frac{a_\alpha^2}{\int q_\alpha^2(x)f_{g_0}(x)\dif\mu(x)}\right)\|m_{k,g}-m_{k,g_0}\|_\infty^2\\
&\ge\left(\min_{k\in[2J]}\frac{1}{d^k}\min_{|\alpha|=k}\frac{a_\alpha^2}{\int q_\alpha^2(x)f_{g_0}(x)\dif\mu(x)}\right)\Delta_g^2, 
\end{aligned}
$$
where $\Delta_g$ is defined as in Lemma \ref{lem:MCL-multi}. Invoking \eqref{eq:MCL-multi}, we then obtain
$$
\begin{aligned}
&\sum_{k=1}^\infty\sum_{|\alpha|=k}c_{\alpha,g}^2
\\
&=\sum_{k=1}^\infty\sum_{|\alpha|=k}\frac{a_\alpha (k+1)^dk^2(m_{\alpha,g}-m_{\alpha,g_0})^2}{\alpha!\chi^2(f_g,f_{g_0})}\\
&\le\sum_{k=1}^\infty\left(\sup_{|\alpha|=k}a_\alpha\right)\frac{(k+1)^dk^2\|m_{k,g}-m_{k,g_0}\|_F^2}{k!\chi^2(f_g,f_{g_0})}\\
&\le\left(\max_{k\in[2J]}d^k\max_{|\alpha|=k}\frac{\int q_\alpha^2(x)f_{g_0}(x)\dif\mu(x)}{a_\alpha^2}\right)\sum_{k=1}^\infty\left(\sup_{|\alpha|=k}a_\alpha\right)\frac{(k+1)^dk^2d^k\|m_{k,g}-m_{k,g_0}\|_2^2}{k!\Delta_g^2}\\
&\le C_1, 
\end{aligned}
$$
where
$$
\begin{aligned}
C_1&\coloneqq\left(\max_{k\in [2J]}d^k\max_{|\alpha|=k}\frac{\int q_\alpha^2(x)f_{g_0}(x)\dif\mu(x)}{a_\alpha^2}\right)\left(\sum_{k=1}^{2J}\left(\sup_{|\alpha|=k}a_\alpha\right)\frac{(k+1)^dk^2d^k}{k!}\right.\\
&+\left.\sum_{k=2J+1}^{\infty}\left(\sup_{|\alpha|=k}a_\alpha\right)\frac{(k+1)^dk^4d^k}{k!}(M+1)^{4Jk}\right)<\infty.  
\end{aligned}
$$
Finally, by the Cauchy-Schwarz inequality, 
$$
\left|\sum_{k=1}^\infty\sum_{|\alpha|=k} c_{\alpha,g}h_\alpha\right|\le\sqrt{\sum_{k=1}^\infty\sum_{|\alpha|=k} c_{\alpha,g}^2}\sqrt{\sum_{k=1}^\infty\sum_{|\alpha|=k} h_\alpha^2}\le\sqrt{C_1}\sqrt{\sum_{k=1}^\infty\sum_{|\alpha|=k} h_\alpha^2}. 
$$
This ensures the existence of an $f_{g_0}\dif\mu$-square integrable envelope for $\cS$. Consequently, assumption \ref{(A3)} is satisfied. 
\end{proof}

\subsection{Proof of Theorem \ref{thm:converge2}}

We apply Theorem \ref{thm:ASY} with $f_0=f_{g_0}$, $\cF=\{f_g:g\in\cG\}$. Assumption \ref{(A1)} is satisfied by definition. Since $f_{g_0}$ is fully supported, \ref{(A2)} can be readily verified. For any $f\in\cF\backslash f_{g_0}$, 
$$
\chi^2(f,f_{g_0})=\sum_{k=1}^Kf_{g_0}(k)\left(\frac{f(k)}{f_{g_0}(k)}-1\right)^2\in(0,\infty). 
$$
To confirm \ref{(A3)}, we note that any $f\in\cF$ can be written as 
$$
f(\cdot)=\left(1+\sum_{k=1}^K\frac{f(k)-f_{g_0}(k)}{f_{g_0}(k)}h_k(\cdot)\right)f_{g_0}(\cdot), 
$$
where $h_k(k')\coloneqq\mathbf{1}_{\{k'=k\}}$ for $1\le k,k'\le K$. This implies that
$$
\cS\subseteq\left\{\frac{\sum_{k=1}^Kc_kh_k}{\|\sum_{k=1}^Kc_kh_k\|_{L_2(f_{g_0}\dif\mu)}}:c\in\bS^{K-1}\right\}. 
$$
Because $f_{g_0}$ is fully supported, we have
$$
\inf_{c\in\bS^{K-1}}\left\|\sum_{k=1}^Kc_kh_k\right\|_{L_2(f_{g_0}\dif \mu)}=\min_{k\in [K]}\sqrt{f_{g_0}(k)}>0. 
$$
This ensures that $\cS$ is, after scaling, a subset of the convex hull of finitely many $f_{g_0}\dif\mu$-square integrable functions $\{h_k,k \in [K]\}$, therefore confirming that assumption \ref{(A3)} is satisfied. \hfill $\Box$

\subsection{Proof of Theorem~\ref{thm:ASY}} \label{sec:proofASY}

We begin by stating and proving several preliminary results.  The ``$\ge$'' direction of \eqref{eq:ASY} crucially depends on the ``star convexity'' of $\cF$ relative to $f_0$, as shown in the following lemma. 

\begin{lemma}\label{lem:LOW}
Under \ref{(A1)} and \ref{(A2)}, for any $s\in\cS$, it holds that
\begin{equation}\label{eq:LOW}
\sup_{f\in\cF}\ell_n(f)-\ell_n(f_0)\ge\frac{1}{2}[(\bG_n(s))_+]^2+o_{\bP}(1). 
\end{equation}
\end{lemma}

\begin{proof} 
By \ref{(A1)}, for any $s=s_f\in\cS$, there is an associated sub-model $\{f_t\}_{t\in[0,\tau]}\subseteq\cF$ given by
$$
f_t\coloneqq\left(1-\frac{t}{\chi(f,f_0)}\right)f_0+\frac{t}{\chi(f,f_0)}f, 
$$
where $\tau\coloneqq\chi(f,f_0) >0$ by \ref{(A2)}. Since
$$
\sup_{f\in\cF}\ell_n(f)-\ell_n(f_0)\ge\sup_{t\in[0,\tau]}\ell_n(f_t)-\ell_n(f_0), 
$$
it suffices to establish a lower bound for the right-hand side. Set $\hat{t}_n\coloneqq(\bG_n(s))_+$. By square integrability of $s$, $\hat{t}_n/\sqrt{n} = O_{\bP}(1/\sqrt{n}) = o_{\bP}(1)$ and thus
\begin{equation}\label{eq:hattn}
\bP(\hat t_n/\sqrt{n} \in [0,\tau]) \to 1.
\end{equation}
By the definition and a Taylor expansion, we get 
$$
\begin{aligned}
\sup_{t\in[0,\tau]}\ell_n(f_t)-\ell_n(f_0)&\ge \ell_n(f_{\frac{\hat{t}_n}{\sqrt{n}}})-\ell_n(f_0)+o_{\bP}(1)\\
&=\sum_{i=1}^n\log\left(1+\frac{\hat{t}_n}{\sqrt{n}}s(X_i)\right)+o_{\bP}(1)\\
&=\frac{\hat{t}_n}{\sqrt{n}}\sum_{i=1}^ns(X_i)-\frac{\hat{t}_n^2}{2n}\sum_{i=1}^ns^2(X_i)+\frac{\hat{t}_n^2}{n}\sum_{i=1}^ns^2(X_i)R\left(\frac{\hat{t}_n}{\sqrt{n}}s(X_i)\right)+o_{\bP}(1)\\
&=\frac{1}{2}[(\bG_n(s))_+]^2+\frac{\hat{t}_n^2}{n}\sum_{i=1}^ns^2(X_i)R\left(\frac{\hat{t}_n}{\sqrt{n}}s(X_i)\right)+o_{\bP}(1), 
\end{aligned}
$$
where $R$ is a deterministic function that satisfies $R(x)\to0$ as $x\to0$. Here, in the first line we used~\eqref{eq:hattn}. In the second line we used that $s_{f_t} = s_f$ for all $t \in [0,\tau]$ and $\|\tfrac{f_t}{f_0} -1\|_{L_2(f_0\dif \mu)} = t$. In the last line we used $\int s^2 f_0 \dif \mu=1$ along with the law of large numbers. For any $\eps>0$, the dominated convergence theorem (DCT) implies that, as $n\to\infty$, 
$$
\bP\left(\frac{1}{\sqrt{n}}\max_{i\in [n]}|s(X_i)|\ge\eps\right)\le n\bP\left(s^2(X_1)\ge n\eps^2\right) \leq \eps^{-2}\int_{\{x: s^2(x)>n\eps^2\}} s^2(x) f_0(x) \dif \mu(x)=o(1). 
$$
Hence, 
\[
\Big| \frac{\hat{t}_n^2}{n}\sum_{i=1}^ns^2(X_i)R\left(\frac{\hat{t}_n}{\sqrt{n}} s(X_i)\right)\Big| \leq \hat{t}_n^2 \Big(\frac{1}{n} \sum_{i=1}^n s^2(X_i)\Big) \max_{i\in [n]} \Big|R\left(\frac{\hat{t}_n}{\sqrt{n}} s(X_i)\right) \Big| = o_{\bP}(1), 
\]
since the argument of $R$ is $o_{\bP}(1)$ uniformly in $i$. Thus the proof is completed. 
\end{proof}

The ``$\le$'' direction of \eqref{eq:ASY} hinges on controlling the chi-square divergence at an $O_{\bP}(1/\sqrt{n})$ rate, as demonstrated by the following lemma.

\begin{lemma}\label{lem:CVR}
Under \ref{(A2)} and \ref{(A3)}, it holds that
\begin{equation}\label{eq:CVR}
\sup_{\text{$\substack{f\in\cF: \\ \ell_n(f)\ge \ell_n(f_0)}$}}\chi(f,f_0)=O_{\bP}\left(\frac{1}{\sqrt{n}}\right). 
\end{equation}
\end{lemma}

\begin{proof} The first part of the proof essentially follows the steps of the proof of Inequality 1.1 in \cite{gassiat2002likelihood}. It is included here for completeness and adapted to our notation. Using the inequality $\log(1+x)\le x-x_-^2/2$ for $x \in (-1,\infty)$, where $x_-\coloneqq\max\{-x, 0\}$, we have for any $f \in \cF$ with $\ell_n(f)-\ell_n(f_0)\ge0$, 
$$
\begin{aligned}
0 \le \ell_n(f)-\ell_n(f_0)&=\sum_{i=1}^n\log\left(1+\chi(f,f_0)s_f(X_i)\right)\\
&\le\chi(f,f_0)\sum_{i=1}^ns_f(X_i)-\frac{1}{2}\chi^2(f,f_0)\sum_{i=1}^n[(s_f(X_i))_-]^2, 
\end{aligned}
$$
where we used he fact that $\chi(f,f_0)s_f(X_i) = \tfrac{f(X_i)}{f_0(X_i)}-1 > -1$. Thus we obtain
$$
\sqrt{n}\sup_{\text{$\substack{f\in\cF: \\ \ell_n(f)\ge \ell_n(f_0)}$}}\chi(f,f_0)\le2\sup_{f\in\cF\backslash f_0}\frac{\frac{1}{\sqrt{n}}\sum_{i=1}^ns_f(X_i)}{\frac{1}{n}\sum_{i=1}^n[(s_f(X_i))_-]^2}\le2\frac{\sup_{s\in\cS}\frac{1}{\sqrt{n}}\sum_{i=1}^ns(X_i)}{\inf_{s\in\cS}\frac{1}{n}\sum_{i=1}^n[(s(X_i))_-]^2}. 
$$
This is essentially Inequality 1.1 in \cite{gassiat2002likelihood} in our notation. The remaining proof follows ideas from the proof of Theorem 2.1 in \cite{gassiat2002likelihood}. 

The numerator in the upper bound is bounded in probability by the Donsker assumption in \ref{(A3)} after noting that $s(X_i)$ are centered. For the denominator, by Example 2.10.7 and Lemma 2.10.14 \citep{van1996weak}, the set $\{(s_-)^2:s\in\cS\}$ is $f_0 \dif \mu$-Glivenko-Cantelli. Moreover, we must have
$$
\inf_{s\in\cS}\int s_-^2 f_0 \dif \mu >0. 
$$
Otherwise, there would exist a sequence $\{s_n\}_{n\in\bN}\subseteq\cS$ with $\int (s_n)_-^2f_0 \dif \mu\to0$. Given $\int (s_n)_+f_0 \dif \mu-\int (s_n)_-f_0 \dif \mu=\int s_nf_0 \dif \mu=0$, $s_n$ converges to zero in $L_1(f_0\dif\mu)$. The envelope assumption in \ref{(A3)} implies that $s_n$ also converges in $L_2(f_0\dif \mu)$, contradicting $\int s_n^2f_0 \dif \mu=1$. As a result, the denominator is bounded away from zero in probability. Combining these observations yields \eqref{eq:CVR}. 
\end{proof}

\begin{proof}[Proof of Theorem \ref{thm:ASY}] 
We begin with the ``$\ge$'' direction of \eqref{eq:ASY}. By the Donsker assumption \ref{(A3)} and the discussion in Section 2.1.2 of \cite{van1996weak} the class $\cS$ is totally bounded in $L_2(f_0\dif \mu)$. Hence, for any $m>0$ we can find a finite $1/m$-net for $\cS$ with respect to this norm, say $\cS_m$. Throughout the proof, we abbreviate $[(\bG_n(s))_+]^2 = (\bG_n(s))_+^2$ to lighten the notation.

Fix an arbitrary $\eps>0$. By the union bound we obtain
$$
\bP\left(\sup_{f\in\cF}\ell_n(f)-\ell_n(f_0)\le\frac{1}{2}\max_{s\in\cS_m}(\bG_n(s))_+^2-\eps\right)
\le
\sum_{s\in\cS_m}\bP\left(\sup_{f\in\cF}\ell_n(f)-\ell_n(f_0)\le\frac{1}{2}(\bG_n(s))_+^2-\eps\right). 
$$
By Lemma \ref{lem:LOW}, the upper bound converges to zero as $n\to\infty$. To obtain the final result, observe the decomposition, 
$$
\begin{aligned}
\bP\left(\sup_{f\in\cF}\ell_n(f)-\ell_n(f_0)\le\frac{1}{2}\sup_{s\in\cS}(\bG_n(s))_+^2-\eps\right)
&\le
\bP\left(\sup_{f\in\cF}\ell_n(f)-\ell_n(f_0)\le\frac{1}{2}\max_{s\in\cS_m}(\bG_n(s))_+^2-\frac{\eps}{2}\right)
\\
&\quad+\bP\left(\frac{1}{2}\max_{s\in\cS_m}(\bG_n(s))_+^2\le\frac{1}{2}\sup_{s\in\cS}(\bG_n(s))_+^2-\frac{\eps}{2}\right). 
\end{aligned}
$$
The first term can be handled by the previous result for $\cS_m$. For the second term, note that
$$
\bP\left(\sup_{s\in\cS_m}(\bG_n(s))_+^2\le\sup_{s\in\cS}(\bG_n(s))_+^2-\eps\right)
\le
\bP\left(\sup_{\text{$\substack{s_1,s_2 \in \cS: \\ \|s_1-s_2
\|_2\le \frac{1}{m}}$}}\left|(\bG_n(s_1))_+^2-(\bG_n(s_2))_+^2\right|\ge\eps\right). 
$$
Next, observe that
\begin{align*}
\sup_{\text{$\substack{s_1,s_2 \in \cS: \\ \|s_1-s_2
\|_2\le\frac{1}{m}}$}} \left| (\bG_n(s_1))_+^2-(\bG_n(s_2))_+^2 \right|
\leq 2\left(\sup_{s \in \cS} |\bG_n(s)|\right) \left(\sup_{\text{$\substack{s_1,s_2 \in \cS: \\ \|s_1-s_2
\|_2\le\frac{1}{m}}$}} |\bG_n(s_1)-\bG_n(s_2)| \right).
\end{align*}
Consequently, by the fact that the sequence $\{\bG_n\}_{n\in\bN}$ is asymptotically uniformly $L_2(f_0\dif \mu)$-equicontinuous in probability (see Example 1.5.10 in \cite{van1996weak}), we obtain
\[
\lim_{m \to \infty}\limsup_{n\to\infty}\bP\left(\sup_{\text{$\substack{s_1,s_2 \in \cS: \\ \|s_1-s_2
\|_2\le \frac{1}{m}}$}}\left|(\bG_n(s_1))_+^2-(\bG_n(s_2))_+^2\right|\ge\eps\right) = 0. 
\]
Thus, the ``$\ge$'' direction of \eqref{eq:ASY} is established. 

To prove the ``$\le$'' direction of \eqref{eq:ASY}, we apply a Taylor expansion argument similar to that in the proof of Lemma \ref{lem:LOW}. Specifically, for any $f \in \cF$
$$
\begin{aligned}
\ell_n(f)-\ell_n(f_0)&=\sum_{i=1}^n\log\left(1+\chi(f,f_0)s_f(X_i)\right)\\
&=\chi(f,f_0)\sum_{i=1}^ns_f(X_i)-\frac{1}{2}\chi^2(f,f_0)\sum_{i=1}^ns_f^2(X_i)\\
&\quad +\chi^2(f,f_0)\sum_{i=1}^ns_f^2(X_i)R\left(\chi(f,f_0)s_f(X_i)\right), 
\end{aligned}
$$
where $R$ is a deterministic function and $R(x)\to0$ as $x\to0$. Let $S$ be an $f_0\dif \mu$-square integrable envelope for $\cS$. By the union bound and the DCT, we have for any fixed $\eps > 0$
\begin{align*}
\bP\left(\frac{1}{\sqrt{n}}\sup_{f\in\cF\backslash f_0}\max_{i\in [n]}|s_f(X_i)|\ge\eps\right)
&\le 
n\bP\left(S^2(X_1)\ge n\eps^2\right)
\\
&\leq \frac{1}{\eps^2}\int_{\{x: S^2(x)>n\eps^2\}} S^2(x) f_0(x) \dif \mu(x) =o(1), 
\end{align*}
as $n\to\infty$. Recall from Lemma \ref{lem:CVR} that $\sup_{f \in \cF: \ell_n(f) \geq \ell_n(f_0)} \chi(f,f_0) = O_{\bP}(1/\sqrt{n})$. Thus, defining 
\[
Y_n \coloneqq \Big(\sup_{\text{$\substack{f\in\cF: \\ \ell_n(f)\ge \ell_n(f_0)}$}} \chi(f,f_0)\Big)\Big( \sup_{f\in\cF\backslash f_0}\max_{i\in [n]}|s_f(X_i)|\Big) = o_{\bP}(1), 
\]
we obtain
$$
\begin{aligned}
&\sup_{\text{$\substack{f\in\cF\backslash f_0: \\ \ell_n(f)\ge \ell_n(f_0)}$}} \left|\chi^2(f,f_0)\sum_{i=1}^ns_f^2(X_i)R\left(\chi(f,f_0)s_f(X_i)\right)\right| 
\\
&\leq  \Big(n\sup_{\text{$\substack{f\in\cF: \\ \ell_n(f)\ge \ell_n(f_0)}$}}\chi^2(f,f_0)\Big) \left(\frac{1}{n} \sum_{i=1}^n S^2(X_i) \right) \sup_{|x| \leq Y_n} |R(x)| = o_{\bP}(1). 
\end{aligned}
$$
Noting further that $\cS^2$ is $f_0\dif\mu$-Glivenko-Cantelli since $\cS$ is $f_0\dif\mu$-Donsker under \ref{(A3)}, see Lemma 2.10.14 in \cite{van1996weak}, we have 
\[
\frac{1}{n} \sum_{i=1}^n s_f^2(X_i) = 1 + o_{\bP}(1), 
\]
uniformly in $f \in \cF$, we obtain
$$
\sup_{f\in\cF}\ell_n(f)-\ell_n(f_0)
=
\sup_{\text{$\substack{f\in\cF\backslash f_0: \\ \ell_n(f)\ge \ell_n(f_0)}$}}
\left(\chi(f,f_0)\sum_{i=1}^ns_f(X_i)-\frac{1}{2}n\chi^2(f,f_0)\right)+o_{\bP}(1). 
$$
Finally, maximizing the term inside the supremum the right-hand side over $\chi(f,f_0) \geq 0$, we obtain the upper bound, 
$$
\sup_{f\in\cF}\ell_n(f)-\ell_n(f_0)\le\frac{1}{2}\sup_{s\in\cS}(\bG_n(s))_+^2+o_{\bP}(1). 
$$
Combined with the lower bound established in the first part of the proof, this completes the argument. 
\end{proof}

\end{document}